\documentclass[preprint,12pt]{elsarticle}

\usepackage{amssymb}
\usepackage{url}
\usepackage{amsthm}
\usepackage{amsmath}

\newtheorem{dfn}{Definition}
\newtheorem{lem}{Lemma}
\newtheorem{obs}{Observation}
\newtheorem{thm}{Theorem}
\newtheorem{prp}{Property}
\newtheorem{pro}{Proposition}
\newtheorem{cor}{Corollary}

\journal{}

\begin{document}

\begin{frontmatter}

\title{Equidistant dimension of Johnson and Kneser graphs}

\author{Jozef Kratica \fnref{mi}}
\ead{jkratica@mi.sanu.ac.rs}
\author{Mirjana \v{C}angalovi\'c \fnref{fon}}
\ead{mirjana.cangalovic@alumni.fon.bg.ac.rs}
\author{Vera Kova\v{c}evi\'c-Vuj\v{c}i\'c \fnref{fon}}
\ead{vera.vujcic@alumni.fon.bg.ac.rs}

 \address[mi]{Mathematical Institute, Serbian Academy of Sciences and Arts, Kneza Mihaila 36/III, 11 000 Belgrade, Serbia}
 \address[fon]{Faculty of Organizational Sciences, University of  Belgrade, Jove Ili\'ca 154, 11000 Belgrade, Serbia}

\begin{abstract}In this paper the recently introduced concept of equidistant dimension
$eqdim(G)$ of graph $G$ is considered. Useful property of distance-equalizer set of
arbitrary graph $G$ has been established.
For Johnson graphs $J_{n,2}$ and Kneser graphs $K_{n,2}$ exact values for
$eqdim(J_{n,2})$ and $eqdim(K_{n,2})$ have been derived, while for Johnson
graphs $J_{n,3}$ it is proved that $eqdim(J_{n,3}) \le n-2$. Finally,
the exact value of $eqdim(J_{2k,k})$ for odd $k$ has been presented.
\end{abstract}


\end{frontmatter}


\section{Introduction and previous work}

The set of vertices $S$ is a resolving (or locating) set of graph $G$ if
all other vertices are uniquely determined by their distances to the
vertices in $S$. The metric dimension of $G$ is the minimum 
cardinality of resolving sets of $G$. Resolving sets for graphs
and the metric dimension were introduced by Slater \cite{metd1}
and, independently, by Harary and Melter \cite{metd2}. 
The concept of doubly resolving set for $G$ has been introduced 
by Caceres et. al \cite{dmetd}.

However, recently, several authors have turned
their attention in the opposite direction from resolvability, thus trying to study
anonymization problems in networks instead of location aspects.
A subset of vertices $A$ is a 2-antiresolving
set for $G$ if, for every vertex $v \notin A$, there exists another different vertex $w \notin A$ such that $v$
and $w$ have the same vector of distances to the vertices of $A$ \cite{adim1}. The 2-metric antidimension
of a graph is the minimum cardinality of 2-antiresolving sets for $G$.
More about this topic can be found in \cite{adim2,adim3}.

In the same spirit, paper \cite{eqdim1} introduces new graph concepts that can also be applied to
anonymization problems in networks: distance-equalizer set and equidistant dimension.
The authors study the equidistant dimension of several classes of graphs,
proving that in the case of paths and cycles this invariant is related
to a classical problem of number theory. They also show that distance-eqalizer sets
can be used for constructing doubly resolving sets, and obtain a new bound
for the minimum cardinality of doubly resolving sets of $G$ in terms of the
metric dimension and the equidistant dimension of $G$. In \cite {eqdimnp} it is proved that the equidistant dimension problem 
is NP-hard in a general case, and eqidistant dimension of lexicographic product of graphs is considered. 

\subsection{Definitions and basic properties}

All graphs considered in this paper are connected, undirected, simple, and finite.
The vertex set and the edge set of a graph $G$ are denoted by $V(G)$ and $E(G)$, respectively.
The order of $G$ is $|V(G)|$. For any vertex $v \in V(G)$, its open neighborhood
is the set $N(v) = \{w \in V(G) | vw \in E(G)\}$ and its closed neighborhood is
$N[v] = N(v) \bigcup \{v\}$. 

The degree of a vertex $v$, denoted by $deg(v)$, is defined as the cardinality of $N(v)$.
If $deg(v) = 1$, then we say that $v$ is a leaf, in which case the only vertex adjacent to
$v$ is called its support vertex.
When $deg(v) = |V(G)|-1$, we say that $v$ is universal. 
The maximum degree of $G$ is
$\Delta(G) = max \{deg(v) | v \in V(G)\}$ and its minimum degree is
$\delta(G) = min \{deg(v) | v \in V(G)\}$.
If all vertices of $G$ have the same degree $r$, i.e. $\Delta(G)=\delta(G)=r$, we say
that graph $G$ is $r$-regular.
The distance between two vertices $v,w \in V(G)$, denoted by $d(v,w)$,
is the leghth of a shortest $u-v$ path,
and the diameter of $G$ is $Diam(G) = max \{d(v,w) | v,w \in V(G)\}$. 
The set of vertices on equal distances from adjacent vertices $u$ and $v$ is denoted in the literature by $_u{W_v}$ (\cite{bal09}).
In general, the same notation can be used also for non-adjacent vertices, i.e. $_u{W_v} = \{ x \in V(G) | d(u,x) = d(v,x)\}$. 

Let $n$ and $k$ be positive integers ($n > k$) and $[n] = \{1,2,...,n\}$.
Then $k$-subsets are subsets of $[n]$ which have cardinality equal to $k$.
The Johnson graph $J_{n,k}$ is an undirected graph defined on all $k$-subsets of set $[n]$ as vertices,
where two $k$-subsets are adjacent if their intersection has cardinality equal to $k-1$.
Mathematically, $V(J_{n,k}) = \{ A | A \subset [n], |A|=k\}$ and
$E(J_{n,k}) = \{ AB | A,B \subset [n], |A|=|B|=k, |A \bigcap B|=k-1\}$.

It is easy to see that $J_{n,k}$ and $J_{n,n-k}$ are isomorphic, so
we shall only consider Johnson graphs with $n \ge 2k$.
The distance between two vertices $A$ and $B$ in $J_{n,k}$
can be computed by Property \ref{distj}.

\begin{prp} \label{distj} For $A,B \in V(J_{n,k})$ it holds 
$d(A,B) = |A \setminus B| = |B \setminus A| = k-|A \bigcap B|$.
\end{prp}

In a special case when $n=2k$ distance between $\overline{A} = V(J_{n,k} \setminus A$ and $B$ can be
computed by Property \ref{comp}.

\begin{prp} \label{comp} For $A,B \in V(J_{2k,k})$ it holds $d(\overline{A},B) = k - d(A,B) = |A \bigcap B|$
\end{prp}
 
Considering  Property \ref{distj}, it is easy to see that Johnson graph $J_{n,k}$ is a $k (n-k)$-regular graph 
of diameter $k$. 

The Kneser graph $K_{n,k}$ is an undirected graph also defined on all $k$-subsets of set $[n]$ as vertices,
where two $k$-subsets are adjacent if their intersection is empty set.
Mathematically, $V(J_{n,k}) = \{ A | A \subset [n], |A|=k\}$ and
$E(J_{n,k}) = \{ AB | A,B \subset [n], |A|=|B|=k, A \bigcap B= \emptyset\}$.

Kneser graph is connected only if $n > k$, it is also $\binom{n-k}{k}$-regular graph.
Specially, for $k=2$, Kneser graph $K_{n,2}$ is the complement of the corresponding 
Johnson graph $J_{n,2}$, and both graphs have diameter 2. Hence, if $d_{K_{n,2}}(A,B)=1$
then $d_{J_{n,2}}(A,B)=2$, and vice versa. Therefore, for $A \ne B$ it holds 
$d_{K_{n,2}}(A,B) = 3 - d_{J_{n,2}}(A,B)$.

\begin{dfn} \label{eqd1} \mbox{\rm(\cite{eqdim1})} Let $u, v, x \in V(G)$. We say that $x$ is equidistant from $u$ and $v$
if $d(u,x) = d(v,x)$. 
\end{dfn}

\begin{dfn} \label{eqd2} \mbox{\rm(\cite{eqdim1})} A subset $S$ of vertices is called a distance-equalizer set for
$G$ if for every two distinct vertices $u, v \in V(G) \setminus S$ there exists a vertex $x \in S$
equidistant from $u$ and $v$. 
\end{dfn}

\begin{dfn} \label{eqd3} \mbox{\rm(\cite{eqdim1})} The equidistant dimension of $G$, denoted by $eqdim(G)$,
is the minimum cardinality of a distance-equalizer set of $G$
\end{dfn}

\begin{prp} \label{eqd4} \mbox{\rm(\cite{eqdim1})} If $v$ is a
universal vertex of a graph $G$, then $S = \{v\}$ is a minimum distance-equalizer set of
$G$, and so $eqdim(G) = 1$.
\end{prp}

\begin{lem} \label{sup1} \mbox{\rm(\cite{eqdim1})} Let $G$ be a graph. If $S$ is a distance-equalizer set of $G$
 and $v$ is a support vertex of $G$, then $S$ contains $v$ or all leaves adjacent to $v$.
\end{lem}

Consequently: 

\begin{cor} \label{sup2} \mbox{\rm(\cite{eqdim1})} $eqdim(G) \ge | \{v \in V(G) | v$ is a support vertex $\} |$.
\end{cor}

\begin{thm} \label{ext1} \mbox{\rm(\cite{eqdim1})} For every graph $G$ of order $n \ge 2$, the following statements hold.
\begin{itemize}
\item $eqdim(G) = 1$ if and only if $\Delta(G) = n - 1$;
\item $eqdim(G) = 2$ if and only if $\Delta(G) = n - 2$.
\end{itemize}
\end{thm}

\begin{cor} \label{ext1a} \mbox{\rm(\cite{eqdim1})}
If $G$ is a graph of order $n$ with $\Delta(G) < n - 2$ then $eqdim(G) \ge 3$.
\end{cor}

\begin{thm} \label{ext2} \mbox{\rm(\cite{eqdim1})} For every graph $G$ of order $n$, the following statements hold.
\begin{itemize}
\item If $n \ge 2$, then $eqdim(G) = n - 1$ if and only if $G$ is a path of order 2;
\item If $n \ge 3$, then $eqdim(G) = n - 2$ if and only if $G \in \{P_3, P_4, P_5, P_6,C_3,C_4,C_5\}$.
\end{itemize}
\end{thm}

\begin{cor} \label{ext2a} \mbox{\rm(\cite{eqdim1})}
If $G$ is a graph of order $n \ge 7$, then $1 \le eqdim(G) \le n-3$.
\end{cor}

\begin{pro} \label{john1} \mbox{\rm(\cite{eqdim1})} For any positive integer $k$, it holds that
$eqdim(J_{n,k}) \le n$ whenever $n \in \{2k - 1, 2k + 1\}$ or $n > 2k^2$.
\end{pro}

In \cite{metdj} the exact value of metric dimension for $J_{n,2}$ for $n \ge 6$ and an upper bound 
of metric dimension for $J_{n,k}$ for $k \ge 3$ are given. 

\section{New results}

\subsection{Some properties of distance-equalizer set of graph $G$}

\begin{lem} \label{wuv1} Let $G$ be a graph. Set $S$ is a distance-equalizer set of $G$ if and only if 
$(\forall u,v \in V(G)) \,\, S \, \bigcap \, (\{u,v\} \, \bigcup \, _u{W_v}) \neq \emptyset$.
\end{lem}
\begin{proof} ($\Rightarrow$) Case 1: $u \in S$\\
Since $u \in S$ and $u \in \{u,v\} \, \bigcup \, _u{W_v}$ then $u$ is also member of their intersection,
i.e. $u \in S \, \bigcap \, (\{u,v\} \, \bigcup \, _u{W_v}) \neq \emptyset$. \\
Case 2: $v \in S$\\
Similarly as in Case 1, since $v \in S$ and $v \in \{u,v\} \, \bigcup \, _u{W_v}$ then $v$ is also member of their intersection,
i.e. $v \in S \, \bigcap \, (\{u,v\} \, \bigcup \, _u{W_v}) \neq \emptyset$. \\
Case 3: $u,v \notin S$\\
Since $S$ is a distance-equalizer set of $G$,
and $u,v \in V(G) \setminus S$ then $(\exists x \in S) \, d(u,x) = d(v,x)$.
Therefore, $x \in S$ and $x \in \, _u{W_v}$ so $S \, \bigcap \, _u{W_v}$ is not empty
(since it contains $x$) implying $S \, \bigcap \, (\{u,v\} \, \bigcup \, _u{W_v}) \neq \emptyset$. \\  

($\Leftarrow$) Let $S \subset V(G)$ and 
$(\forall u,v \in V(G)) \,\, S \, \bigcap \, (\{u,v\} \, \bigcup \, _u{W_v}) \neq \emptyset$.
Suppose that $u,v \in V(G) \setminus S$. From $\emptyset \ne S \, \bigcap \, (\{u,v\} \, \bigcup \, _u{W_v}) =$ \\
$ (S \bigcap \, (\{u,v\}) \bigcup \, (S \bigcap \, _u{W_v}) = S \bigcap \, _u{W_v}$. It follows that 
there exists $x \in S$ such that $d(u,x) = d(v,x)$, i.e. $S$ is a distance-equalizer set of $G$.
\end{proof}

\begin{cor} \label{wuv2} Let $G$ be a graph, and $u$ and $v$ any vertices from $V(G)$.
If $S$ is a distance-equalizer set of $G$ and $_u{W_v} = \emptyset$ then $u \in S$ or $v \in S$.
\end{cor}

It should be noted that Lemma \ref{wuv1} from \cite{eqdim1} is a consequence of 
Corollary \ref{wuv2}. Indeed, if $v$ is a support vertex of $G$ and $u$ is one of 
leaves adjacent to $v$, it is obvious that $(\forall x \in V(G) \setminus \{u\}) d(u,x) = d(v,x) + 1$
and $1 = d(u,v) = d(u,u)+1$,  and, therefore, $_u{W_v} = \emptyset$. 
If $S$ is a distance-equalizer set of $G$, by Corollary \ref{wuv2},
$S$ contains $v$ or all leaves adjacent to $v$. 

\subsection{Equidistant dimension of $J_{n,2}$}

The exact value of $eqdim(J_{n,2})$ for $n \ge 4$ is given by Observation \ref{john2a} and Theorem \ref{john2}.

\begin{obs} \label{john2a} By a total enumeration, it is found that 
\begin{itemize}
\item $eqdim(J_{4,2}) = 2$ with the corresponding distance-equalizer set $S = \{ \{1,2\}, \{3,4\}\}$;
\item $eqdim(J_{5,2}) = 3$ with the corresponding distance-equalizer set $S = \{ \{1,2\}, \{1,3\}, \{2,3\}\}$.
\end{itemize}
\end{obs}

\begin{thm} \label{john2} For $n \ge 6$ it holds  $eqdim(J_{n,2}) = 3$.
\end{thm}
\begin{proof}Step 1: $eqdim(J_{n,2}) \ge 3$\\
Since $J_{n,k}$ is $k \cdot (n-k)$-regular graph,
so $\Delta(J_{n,k}) = \delta(J_{n,k}) = k \cdot (n-k)$.
For $k=2$ it follows that $\Delta(J_{n,2}) =  2 \cdot (n-2)$.
Since $|J_{n,2}| = \binom{n}{2} = \frac{n \cdot (n-1)}{2}$
it is obvious that for $n \ge 5$ it holds 
$\Delta(J_{n,2}) =  2 \cdot (n-2) < |J_{n,2}|-2 = \frac{n \cdot (n-1)}{2} - 2$,
so by Corollary \ref{ext1a} it follows that $eqdim(J_n) \ge 3$.

Step 2: $eqdim(J_{n,2}) \le 3$\\
Let $S = \{ \{1,2\}, \{1,3\}, \{2,3\}  \}$. We will prove that set $S$
is a distance-equalizer set by checking all pairs of vertices $X$ and $Y$ from $V(J_{n,2}) \setminus S$.

Case 1: $\{1,2,3\} \bigcap X = \emptyset$ and $\{1,2,3\} \bigcap Y = \emptyset$ \\
Let $Z = \{1,2\}$. Then 
$d(X,Z)= 2-|X \bigcap Z| = 2 = 2-|Y \bigcap Z| = d(Y,Z)$.

Case 2: $\{1,2,3\} \bigcap X = \emptyset$ and $\{1,2,3\} \bigcap Y \ne \emptyset$ \\
Since $Y \notin S$ then $|Y \bigcap \{1,2,3\}|  = 1$. Let $Z = \{1,2,3\} \setminus Y$. 
It is obvious that $Z \subset \{1,2,3\}$ and $|Z| = 2$ implying $Z \in S$.
Since $\{1,2,3\} \bigcap X = \emptyset$ and $Y \bigcap Z = \emptyset$ 
then $d(X,Z)= 2-|X \bigcap Z| = 2 = 2-|Y \bigcap Z| = d(Y,Z)$.

Case 3: $\{1,2,3\} \bigcap X \ne \emptyset$ and $\{1,2,3\} \bigcap Y = \emptyset$ \\
This case is analogous as Case 2, only swap sets X and Y.

Case 4: $\{1,2,3\} \bigcap X \ne \emptyset$ and $\{1,2,3\} \bigcap Y \ne \emptyset$ and $X \bigcap Y \bigcap \{1,2,3\} = \emptyset$\\
Let $Z = \{1,2,3\} \bigcap (X \bigcup Y)$.
It is obvious that $Z \subseteq \{1,2,3\}$.
Since $X,Y \notin S$ then $|X \bigcap \{1,2,3\}|  = 1$ and $|Y \bigcap \{1,2,3\}|  = 1$
it holds $|Z|=2$ so  $Z \in S$. 
Therefore, $d(X,Z)= 2-|X \bigcap Z| = 1 = 2-|Y \bigcap Z| = d(Y,Z)$.

Case 5: $\{1,2,3\} \bigcap X \ne \emptyset$ and $\{1,2,3\} \bigcap Y \ne \emptyset$ and $X \bigcap Y \bigcap \{1,2,3\} \ne \emptyset$\\
Since $X,Y \notin S$ it holds $|X \bigcap Y \bigcap \{1,2,3\}|=1$.
Let $Z = \{1,2,3\} \setminus X$. It is obvious that $Z = \{1,2,3\} \setminus Y$ and
$X \bigcap Z = Y \bigcap Z = \emptyset$.
Therefore, $d(X,Z)= 2-|X \bigcap Z| = 2 = 2-|Y \bigcap Z| = d(Y,Z)$.  
 
Since $(\forall X,Y \in V(J_{n,2}) \setminus S) (\exists Z \in S) d(X,Z) = d(Y,Z)$, then
$S$ is a distance-equalizer set for $J_{n,2}$ and thus  $eqdim(J_{n,2}) \le |S| = 3$.
From Step 1 and Step 2 it holds $eqdim(J_{n,2}) = 3$ for all $n \ge 6$.
\end{proof}

\subsection{An upper bound of equidistant dimension of $J_{n,3}$}

The next theorem gives a tight upper bound of $eqdim(J_{n,3})$ for $n \ge 9$.
The remaining cases when $n \in \{6,7,8\}$ are resolved by Theorem \ref{john2kk}
for $n=6$ and Table \ref{eqjk3} for $n=7$ and $n=8$.

\begin{thm} \label{john3} For $n \ge 9$ it holds  $eqdim(J_{n,3}) \le n-2$.
\end{thm}
\begin{proof}
Let $S = \{ \{1,2,j\} | 3 \le j \le n\}$. It can be proved that set $S$
is a distance-equalizer set for $J_{n,3}$, i.e. for each two vertices $X$
and $Y$ from $V(J_{n,3}) \setminus S$,
there exists a vertex $Z=\{1,2,l\}$ from $S$, such that $d(X,Z) = d(Y,Z)$.
We will consider four cases:

Case 1: $\{1,2\} \bigcap X = \emptyset$ and $\{1,2\} \bigcap Y = \emptyset$ \\
It is easy to see that $|\{1,2\} \bigcup X \bigcup Y| \le 8$.
As $n \ge 9$, then there exists $l \in \{3,4,...,n\}$ such that $l \notin X \bigcup Y$. 
Now, for vertex $Z=\{1,2,l\}$ from $S$, 
$d(X,Z)= 3-|X \bigcap Z| = 3 = 3-|Y \bigcap Z| = d(Y,Z)$.

Case 2: $\{1,2\} \bigcap X \ne \emptyset$ and $\{1,2\} \bigcap Y \ne \emptyset$ \\
As $X \notin S$ and $Y \notin S$, then $|\{1,2\} \bigcap X| = 1$
and $|\{1,2\} \bigcap Y| = 1$ and, consequently, $|\{1,2\} \bigcup X \bigcup Y| \le 6$.
As $n \ge 9$, then there exists $l \in \{3,4,...,n\}$ such that $l \notin X \bigcup Y$.
Now, for vertex $Z=\{1,2,l\}$ from $S$, 
$d(X,Z)= 3-|X \bigcap Z| = 2 = 3-|Y \bigcap Z| = d(Y,Z)$.

Case 3: $\{1,2\} \bigcap X \ne \emptyset$ and $\{1,2\} \bigcap Y = \emptyset$ \\
As $X \notin S$, then $|\{1,2\} \bigcap X| = 1$ and, consequently,
$(Y \setminus X) \bigcap \{1,2\} = \emptyset$ and $|Y \setminus X| \ge 1$. 
It means that there exists $l \in \{3,4,...,n\}$ such that $l \notin Y \setminus X$.
Now, for vertex $Z=\{1,2,l\}$ from $S$, 
$d(X,Z)= 3-|X \bigcap Z| = 2 = 3-|Y \bigcap Z| = d(Y,Z)$. 

Case 4: $\{1,2\} \bigcap X = \emptyset$ and $\{1,2\} \bigcap Y \ne \emptyset$ \\
This case can be reduced to Case 3.
 
Based on all previous cases, for each pair of vertices from $V(J_{n,3}) \setminus S$
there exists a vertex $Z \in S$ such that $d(X,Z) = d(Y,Z)$.
Therefore, set $S$ is a distance-equalizer set for $J_{n,3}$. 
As $|S| = n-2$, then $eqdim(J_{n,3}) \le |S| = n-2$.
\end{proof}

\subsection{Equidistant dimension of $J_{2k,k}$, for odd $k$}

Since $\binom{2k}{k}$ is even, then it is possible to make a partitition $(P_1,P_2)$ of $V(J_{2k,k})$, 
such that $P_1 \bigcap P_2 = \emptyset$, $P_1 \bigcup P_2 = V(J_{2k,k})$  and $|P_1| = |P_2| = \frac{1}{2} \cdot \binom{2k}{k}$.
In the sequel we will use the following partition:
$P_1 = \{X \in V(J_{2k,k}) \,:\, |X \bigcap \{1,2,...,k\}| > |X \bigcap \{k+1,k+2,...,2k\}|\}$,
and $P_2 = V(J_{2k,k}) \setminus P_1$. It shoud be noted that 
for odd $k$ it holds $|X \bigcap \{1,2,...,k\}| \ne |X \bigcap \{k+1,k+2,...,2k\}|$,
so $P_2 = \{X \in V(J_{2k,k}) \,:\, |X \bigcap \{1,2,...,k\}| < |X \bigcap \{k+1,k+2,...,2k\}|\}$
and, consequently, $|P_1| = |P_2| = \frac{1}{2} \cdot \binom{2k}{k}$. 

\begin{thm} \label{john2kk} For any odd $k \ge 3$ it holds $eqdim(J_{2k,k}) = \frac{1}{2} \cdot \binom{2k}{k}$.
\end{thm}
\begin{proof}Step 1: $eqdim(J_{2k,k}) \ge \frac{1}{2} \cdot \binom{2k}{k}$\\
Let us consider $\frac{1}{2} \cdot \binom{2k}{k}$ pairs of vertices $(X,Y)$ from $V(J_{2k,k})$,
such that $X \in P_1$ and $Y = [2k] \setminus X \in P_2$. 
For any vertex $Z \in V(J_{2k,k})$ it holds $|Z \bigcap X| + |Z \bigcap Y| = k$.
Since $k$ is odd, $|Z \bigcap X|$ is odd and $|Z \bigcap Y|$ is even,
or vice versa. Therefore, $|Z \bigcap X| \ne |Z \bigcap Y|$ implying
$d(X,Z) = k-|Z \bigcap X| \ne k - |Z \bigcap Y| = d(Y,Z)$,
so $_{X}W_Y = \emptyset$. According to Corollary \ref{wuv2},
if $S$ is a distance-equalizer set for graph $J_{2k,k}$
then either $X \in S$ or $Y \in S$, for each pair $(X,Y)$.
Since the number of pairs is $\frac{1}{2} \cdot \binom{2k}{k}$,
then $|S| \ge \frac{1}{2} \cdot \binom{2k}{k}$.

Step 2: $eqdim(J_{2k,k}) \le \frac{1}{2} \cdot \binom{2k}{k}$\\
We shall prove that $P_1$ is a distance-equalizer set for $J_{2k,k}$.
For any two vertices $Y$ and $Z$ from $P_2 = V(J_{2k,k}) \setminus P_1$, 
let us construct $X \in P_1$ such that $d(Y,X) = d(Z,X)$.
Since $|Y| = |Z| = k$ it follows that 
$|Y \setminus Z| = |Y| - |Y \bigcap Z| = |Z| - |Y \bigcap Z| = |Z \setminus Y|$.
Additionally, as $Y,Z \in V(J_{2k,k})$ then $|Y \bigcap Z| = |\overline{Y} \bigcap \overline{Z}|$.
Let $U_1 = (Y \bigcap Z) \bigcup (\overline{Y} \bigcap \overline{Z})$. It is easy to see
that $U_1 \bigcap Y = U_1 \bigcap Z$ and $|U_1|$ is even so $k+1-|U_1|$ is 
also even. \\
Case 1. If $|U_1| < k$ let $a \in U_1$ be arbitrary index and $U_2 = U_1 \setminus \{a\}$.
Let $W_1$ and $W_2$ be any subsets of $Y \setminus Z$ and $Z \setminus Y$ of cardinality
$\frac{k+1-|U_1|}{2}$ elements, respectively. Now let $U_3 = U_2 \bigcup W_1 \bigcup W_2$.
It is obvious that $W_1 \subset Y$, $W_1 \bigcap Z = \emptyset$, $W_2 \subset Z$, $W_2 \bigcap Z = \emptyset$.
Moreover, $|W_1| = |W_2|$, and therefore 
$|U_3 \bigcap Y| = |U_2 \bigcap Y| + |W_1 \bigcap Y| = |U_2 \bigcap Y| + |W_1| =$
$|U_2 \bigcap Z| + |W_2| = |U_2 \bigcap Z| + |W_2 \bigcap Z| = |U_3 \bigcap Z|$. \\
Case 2. If $|U_1| > k$ let $U_3$ be any subset of $U_1$ of cardinality $k$. 
It is obvious that $U_3 \subset (Y \bigcap Z) \bigcup (\overline{Y} \bigcap \overline{Z})$
so $|U_3 \bigcap Y| = |U_3 \bigcap Z|$. \\

In both cases $|U_3| = k$ so $U_3 \in V(J_{2k,k})$.
Therefore, in both cases $|U_3 \bigcap Y| =|U_3 \bigcap Z|$ and hence
$d(U_3,Y) = k-|U_3 \bigcap Y| = k-|U_3 \bigcap Z| = d(U_3,Z)$.

Finally, we construct $X$ as follows.
If $U_3 \in P_1$ then $X = U_3$. Otherwise,
if $U_3 \in P_2$ then $X = \overline{U_3} \in P_1$,
and by Property \ref{comp} it holds $d(\overline{U_3},Y) = k-d(U_3,Y) = |U_3 \bigcap Y| =$
$|U_3 \bigcap Z| = k-d(U_3,Z) = d(\overline{U_3},Z)$.
As, $d(Y,X) = d(Z,X)$ and $X \in P_1$,
it follows that $P_1$ is a distance-equalizer set for graph $J_{2k,k}$.
Therefore, $eqdim(J_{2k,k}) \le |P_1| = \frac{1}{2} \cdot \binom{2k}{k}$. 
\end{proof}

\subsection{Equidistant dimension of $K_{n,2}$}

Exact value for $eqdim(K_{n,2})$ is given in Theorem \ref{knes1},
and it is equal to $eqdim(J_{n,2}) = 3$.

\begin{thm} \label{knes1} $eqdim(K_{n,2}) = 3$.
\end{thm}
\begin{proof}Step 1: $eqdim(K_{n,2}) \ge 3$\\
As stated in Section 1, Kneser graph $K_{n,k}$
exists only for $n > 2 \cdot k$ implying that
for $k=2$ all Kneser graphs $K_{n,2}$ satisfy $n \ge 5$. 
Similarly as for Johnson graphs, Kneser graph $K_{n,k}$ is $\binom{n-k}{k}$-regular graph,
so $\Delta(K_{n,k}) = \delta(K_{n,k}) = \binom{n-k}{k}$.
For $k=2$ it follows that $\Delta(K_{n,2}) =  \frac{(n-2)(n-3)}{2}$.
Since $|K_{n,2}| = \binom{n}{2} = \frac{n \cdot (n-1)}{2}$
it is obvious that for $n \ge 5$ it holds $4 \cdot n > 10$ so $(n-2)(n-3) = n^2-5n+6 < n^2-n-4 = n(n-1)-4$
implying $\binom{n-2}{2} < \binom{n}{2}-2$ which means 
$\Delta(K_{n,2}) =  \binom{n-2}{2} < |K_{n,2}|-2 = \binom{n}{2}-2$,
so by Corollary \ref{ext1a} it follows that $eqdim(K_{n,2}) \ge 3$.

Step 2: $eqdim(J_{n,2}) \le 3$\\
As already noticed $\overline{J_{n,2}} = K_{n,2}$ and $Diam(J_{n,2}) = Diam(K_{n,2}) = 2$ so $V(J_{n,2}) = V(K_{n,2})$
and for each two vertices $A,B \in V(K_{n,2})$ with $A \ne B$ it holds $d(A,B) = 3 - d_{J_{n,2}}(A,B)$,
where $d(A,B)$ and  $d_{J_{n,2}}(A,B)$ are distances between $A$ and $B$ in Kneser graph $K_{n,2}$ and
Johnson graph $J_{n,2}$, respectivelly.
 
Let $S = \{ \{1,2\}, \{1,3\}, \{2,3\}  \}$, and $X$ and $Y$ are any vertices from 
$V(K_{n,2}) \setminus S$. Since $V(J_{n,2}) = V(K_{n,2})$, and by Theorem \ref{john2}, the same set $S = \{ \{1,2\}, \{1,3\}, \{2,3\}  \}$ is proved 
to be a distance-equalizer set for graph $J_{n,2}$, then $(\forall X,Y \in V(J_{n,2}) \setminus S) (\exists Z \in S) \, d_{J_{n,2}}(X,Z) = d_{J_{n,2}}(Y,Z)$.
It follows that $d(X,Z) = 3 - d_{J_{n,2}}(X,Z) = 3 - d_{J_{n,2}}(Y,Z) = d(Y,Z)$.
Therefore, the same set $S$ is also a distance-equalizer set for graph $K_{n,2}$. 
From Step 1 and Step 2 it holds $eqdim(K_{n,2}) = 3$.
\end{proof}

\subsection{Some other individual exact values}

It is interesting to examine values of $eqdim(J_{n,k})$ and $eqdim(K_{n,k})$
in cases which are not covered by the obtained theoretical results presented above.
Table \ref{eqjk3} contains such values for Johnson and Kneser graphs up to 84 vertices
obtained by a total enumeration. Since Kneser graphs are not connected for $n=2k$,
graph $K_{8,4}$ is not connected, which is denoted by ''-''. 

\begin{table}
\caption{$eqdim(J_{n,k})$ and $eqdim(K_{n,k})$ for $k \ge 3$} \label{eqjk3}
 \small
\begin{center} 
\begin{tabular}{|c|c|c|c|}
\hline
$n$ & $k$ &  $eqdim(J_{n,k})$ &  $eqdim(K_{n,k})$\\
\hline
7 & 3 & 5  & 5 \\
\hline
8 & 3 & 8 & 3\\
8 & 4 & 7 & - \\
\hline
9 & 3 & 7 & 3 \\
 \hline
\end{tabular}
\end{center}
\end{table}

\section{Conclusions}

In this paper, equidistant dimensions of Johnson and Kneser graphs are considered.
Exact values $eqdim(J_{n,2})=3$, $eqdim(J_{2k,k})=\frac{1}{2} \cdot \binom{2k}{k}$ for odd $k$ and 
$eqdim(K_{n,2})=3$ are found.
Moreover, it is proved that $n-2$ is a tight upper bound for $eqdim(J_{n,3})$.

Further work can be directed to finding equidistant dimension of other interesting classes of graphs.
Also, it would be interesting to develop exact and/or heuristic approaches for solving 
equidistant dimension problem.

\bibliographystyle{elsarticle-num}
 \bibliography{paper}

\end{document}